\documentclass[submission]{dmtcs}

\usepackage{amsmath}

\usepackage[latin1]{inputenc}
\usepackage[english]{babel}

\newtheorem{theorem}{Theorem}
\newtheorem{definition}{Definition}
\newtheorem{lemma}{Lemma}
\newtheorem{proposition}{Proposition}
\newtheorem{remark}{Remark}

\author{Mathias P\'etr\'eolle \addressmark{1}\thanks{Email: \email{petreolle@math.univ-lyon1.fr}. }
}  
\title{A Nekrasov-Okounkov type formula for $\widetilde{C}$}
\address{\addressmark{1} Institut Camille Jordan, Lyon, France}
\received{14th November 2014}
\revised{?}
\accepted{?}
\begin{document}
\maketitle

\begin{abstract}
\paragraph{Abstract.} In 2008, Han rediscovered an expansion of powers of Dedekind $\eta$ function due to Nekrasov and Okounkov by using Macdonald's identity in type $\widetilde{A}$. In this paper, we obtain new combinatorial expansions of powers of $\eta$, in terms of partition hook lengths, by using Macdonald's identity in type $\widetilde{C}$ and a new bijection. As applications, we derive a symplectic hook formula and a relation between Macdonald's identities in types $\widetilde{C}$, $\widetilde{B}$, and $\widetilde{BC}$.

\paragraph{R\'esum\'e.} En 2008, Han a red\'ecouvert un d\'eveloppement des puissances de la fonction $\eta$ de Dedekind, d\^u \`a Nekrasov et Okounkov, en utilisant l'identit\'e de Macdonald en type $\widetilde{A}$. Dans cet article, nous obtenons un nouveau d\'eveloppement combinatoire des puissances de $\eta$, en termes de longueurs d'\'equerres de partitions, en utilisant l'identit\'e de Macdonald en type $\widetilde{C}$ ainsi qu'une nouvelle bijection. Plusieurs applications en sont d\'eduites, comme un analogue symplectique de la formule des \'equerres, ou une relation entre les identit\'es de Macdonald en types $\widetilde{C}$, $\widetilde{B}$ et $\widetilde{BC}$.

\end{abstract}

\section{Introduction}
Recall the Dedekind $\eta$ function, which is a weight $1/2$ modular form, defined as follows:
\begin{equation}
\eta(x) =x^{1/24} \prod _{k \geq 1} (1-x^k),
\end{equation}
where $|x| < 1$ (we will assume this condition all along this paper). Apart from its modular properties, due to the factor $x^{1/24}$, this function plays a fundamental role in combinatorics, as it is related to the generating function of integer partitions. Studying expansions of powers of $\eta$ is natural, in the sense that it yields a certain amount of interesting questions both in combinatorics and number theory, such as Lehmer's famous conjecture (see for instance \cite{JPS}).
In 2006, in their study of the theory of Seiberg-Witten on supersymetric gauges in particle physics \cite{NO}, Nekrasov and Okounkov obtained the following formula:
\begin{equation}\label{nekrasov}
\prod_{k \geq 1} (1-x^k)^{z-1} = \sum_{\lambda \in \mathcal{P}} x^{|\lambda|}  \prod_{h\in \mathcal{H}(\lambda)} \left( 1-\frac{z}{h^2} \right),
\end{equation}
where $\mathcal{P}$ is the set of integer partitions and $\mathcal{H}(\lambda)$ is the multiset of hook lengths of $\lambda$ (see Section~\ref{section1} for precise definitions).
In 2008, this formula was rediscovered and generalized by Han \cite{HAN}, through two main tools, one arising from an algebraic context and the other from a more combinatorial one. From this result, Han derived many applications in combinatorics and number theory, such as the marked hook formula or a reformulation of Lehmer's conjecture. Formula \eqref{nekrasov} was next proved and generalized differently by Iqbal \emph{et al.} in 2013 \cite{IEA} by using plane partitions, Cauchy's formula for Schur function and the notion of topological vertex.

The proof of Han uses on the one hand a bijection between $t$-cores and some vectors of integers, due to Garvan, Kim and Stanton in their proof of Ramanujan's congruences \cite{GKS}. Recall that $t$-cores had originally been introduced in representation theory to study some characters of the symmetric group \cite{GK}. On the other hand, Han uses Macdonald's identity for affine root systems. It is a generalization of Weyl's formula for finite root systems $R$ which can itself be written as follows:
\begin{equation}
\prod _{\alpha >0} (e^{\alpha/2} - e^{-\alpha/2}) = \sum_{w \in W(R)} \varepsilon (w) e^{w \rho},
\end{equation}
where the sum is over the elements of the Weyl group $W(R)$, $\varepsilon$ is the sign, and $\rho$ is an explicit vector depending on $W(R)$. Here, the product ranges over the positive roots $R^+$, and the exponential is formal. Macdonald specialized his formula for several affine root systems and exponentials. In all cases, when the formal exponential is mapped to the constant function equal to 1, the product side can be rewritten as an integral power of Dedekind $\eta$ function. In particular, the specialization used in~\cite{HAN} corresponds  to the type $\widetilde{A}_t$, for an odd positive integer $t$, and reads (here $\|.\|$ is the euclidean norm):
\begin{equation}\label{equaA}
\eta(x)^{t^2-1} =c_0 \sum_{{\bf v}} x^{\|{\bf v}\|^2 /2t}\prod_{i<j}(v_i-v_j) , \quad \mbox{with~} c_0:= \frac{(-1)^{(t-1)/2}}{1! 2! \cdots (t-1)!},
\end{equation}
where the sum is over $t$-tuples ${\bf v} :=(v_0,\ldots, v_{t-1}) \in \mathbb{Z}^{t}$ such that $v_i \equiv i \mbox{~mod~} t$ and $v_0+\cdots +v_{t-1}=0$. Han next uses a refinement of the aforementioned bijection to transform the right-hand side into a sum over partitions,  and proves \eqref{nekrasov} for all odd integers $t$. Han finally transforms the right-hand side through very technical considerations to show that \eqref{nekrasov} is in fact true for all complex number $t$. A striking remark is that the factor of modularity $x^{(t^2-1)/24}$ in $\eta(x)^{t^2-1}$ cancels naturally in the proof when the bijection is used. 
\smallskip

This approach immediatly raises a question, which was asked by Han in~\cite{HAN}: can we use specializations of Macdonald's formula for other types to find new combinatorial expansions of the powers of $\eta$? In the present paper, we give a positive answer for the case of type  $\widetilde{C}$ and, as shall be seen later, for types $\widetilde{B}$ and $\widetilde{BC}$. In the case of type $\widetilde{C}_t$, for an integer $t \geq 2$, Macdonald's formula reads:
\begin{equation}\label{equaC}
\eta(x)^{2t^2+t} = c_1 \sum_{{\bf v}} x^{\|{\bf v}\|^ 2/(4t+4)} \prod_i v_i \prod_{i<j}(v_i^2-v_j^2) ,\quad \mbox{with~} c_1:= \frac{(-1)^{\lfloor t/2 \rfloor}}{1! 3! \cdots (2t+1)!},
\end{equation}
where the sum ranges over $t$-tuples ${\bf  v} :=(v_1,\ldots, v_t) \in \mathbb{Z}^t$ such that $v_i \equiv i \mbox{~mod~} 2t+2$. The first difficulty in providing an analogue of \eqref{nekrasov} through \eqref{equaC} is to find which combinatorial objects should play the role of the partitions $\lambda$. Our main result is the following possible answer.

\begin{theorem}\label{theoremeintro} For any complex number $t$, with the notations and definitions of Section~\ref{section1}, the following expansion holds:
\begin{equation}\label{eqtheoremeintro}
\prod_{k \geq 1}(1-x^k) ^{2t^2+t} = \sum_{\lambda \in DD}\delta_\lambda\,   x^{|\lambda|/2} \prod_{h \in  \mathcal{H}(\lambda)} \left( 1- \frac{2t+2}{h\, \varepsilon_h }\right),
\end{equation}
where the sum is over doubled distinct partitions, $\delta_\lambda$ is equal to $1$ (\emph{resp.} $-1$) if the Durfee square of $\lambda$ is  of even (\emph{resp.} odd) size, and $\varepsilon_h$ is equal to $-1$ if $h$ is the hook length of a box strictly above the diagonal and to $1$ otherwise.
\end{theorem}
 To prove this, we will use~\eqref{equaC} and a bijection obtained through results of~\cite{GKS}. Many applications can be derived from Theorem~\ref{theoremeintro}, which we are able to generalize with more parameters as did Han for \eqref{nekrasov}. However, we will only highlight two consequences, a combinatorial one and a more algebraic one. The first is the following symplectic analogue of the famous hook formula (see for instance \cite{EC}), valid for any positive integer $n$:
\begin{equation}\label{hookf}
\sum_{\stackrel{\lambda \in DD}{  |\lambda|= 2n} } \prod_{h \in \mathcal{H}(\lambda)}\frac{1}{h} = \frac{1}{2^n n!}.
\end{equation}
The second, which is expressed in Theorem~\ref{equi}, is a surprising link between the family of Macdonald's formulas in type $\widetilde{C}_t$ (for all integers $t \geq 2$), the one in type $\widetilde{B}_t$ (for all integers $t \geq 3$), and the one in type $\widetilde{BC}_t$ (for all integers $t \geq 1$).

The paper is organized as follows. In Section~\ref{section1}, we recall the definitions and notations regarding partitions, $t$-cores, self-conjugate and doubled distinct partitions. Section~\ref{section2} is devoted to  sketching the proof of Theorem~\ref{theoremeintro} using a new bijection between the already mentioned subfamilies of partitions and some vectors of $\mathbb{Z}^{t}$, and its properties that we will explain. In Section~\ref{section3}, we derive some applications from Theorem~\ref{theoremeintro}, such as the symplectic hook formula \eqref{hookf}, and the connection between \eqref{equaC} and Macdonald's identities in types $\widetilde{B}$ and $\widetilde{BC}$, which are shown in Theorem~\ref{equi} to be all generalized by Theorem~\ref{theoremeintro}.

\section{Integer partitions and $t$-cores}\label{section1}
In all this section, $t$ is a fixed positive integer.
\subsection{Definitions}\label{defs}
We recall the following definitions, which can be found in \cite{EC}. A \emph{partition} $\lambda=(\lambda_1,\lambda_2,\ldots, \lambda_\ell)$ of the integer $n \geq 0$ is a finite non-increasing sequence of positive integers whose sum is $n$. The $\lambda_i$'s are the parts of $\lambda$, $\ell := \ell(\lambda)$ is its length, and $n$ its weight, denoted by $|\lambda|$. Each partition can be represented by its Ferrers diagram as shown in Figure~\ref{fig1}, left. (Here we represent the Ferrers diagram in French convention.)

\begin{figure}[h!] \begin{center}
\includegraphics[scale=1.2]{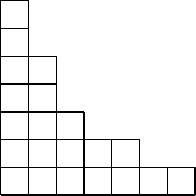}
\hspace*{1cm}
\includegraphics[scale=1.2]{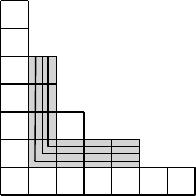}
\hspace*{1cm}
\includegraphics[scale= 1.2]{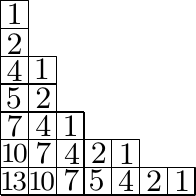}
\caption{\label{fig1} The Ferrers diagram of the partition $(7,5,3,2,2,1,1)$, a principal hook and the hook lengths.}
\end{center}
\end{figure}

For each box $v=(i,j)$ in the Ferrers diagram of $\lambda$ (with $i \in \{1,\ldots, \ell\}$ and $j \in \{1,\ldots,\lambda_i \}$), we define the \emph{hook of $v$} as the set of boxes $u$ such that either $u$ lies on the same row and above $v$, or $u$ lies on the same column and on the right of $v$. The \emph{hook length} of $v$ is the cardinality of the hook of $v$ (see Figure~\ref{fig1}, right).
The hook of $v$ is called \emph{principal} if $v=(i,i)$ (\textit{i.e.}  $v$ lies on the diagonal of $\lambda$, see Figure~\ref{fig1}, center). The \emph{Durfee square} of $\lambda$ is the greatest square included in its Ferrers diagram, the length of its side is the Durfee length, denoted by $D(\lambda)$: it is also the number of principal hooks. We denote by $\delta_\lambda$ the number $(-1)^{D(\lambda)}$.
\begin{definition}Let $\lambda $ be a partition. We say that $\lambda$ is a \emph{ $t$-core} if and only if no hook length of $\lambda$ is a multiple of $t$. 
\end{definition}
Recall \cite{HAN} that $\lambda$ is a $t$-core if and only if no hook length of $\lambda$ is equal to $t$. We denote by $\mathcal{P}$ the set of partitions and by  $\mathcal{P}_{(t)}$ the subset of $t$-cores.
\begin{definition}Let  $\lambda $ be a partition. The \emph{$t$-core of $\lambda$} is the partition $T(\lambda)$ obtained from $\lambda$  by removing in its Ferrers diagram all the ribbons of length $t$, and by repeating this operation until we can not remove any ribbon. 
\end{definition}

\begin{figure}[h!]
\begin{center}
\includegraphics[scale=1]{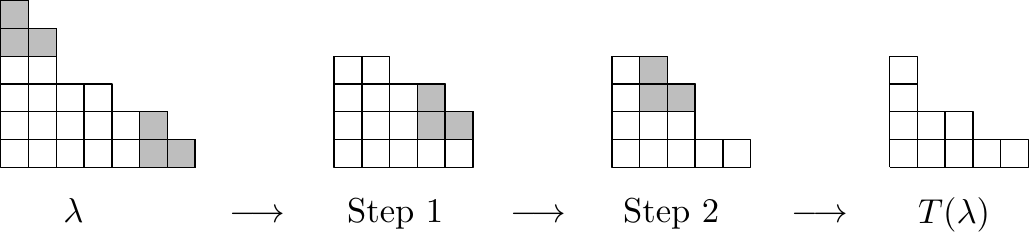}
\caption{\label{fig2} The construction of the $3$-core of the partition $\lambda=(7,6,4,2,2,1)$. In grey, the deleted ribbons.}
\end{center}
\end{figure}

The definition of ribbons can be found in \cite{EC}, see Figure~\ref{fig2} for an example. Note that $T(\lambda)$ does not depend on the order of removal (see \cite[p. 468]{EC} for a proof). In particular, as a ribbon of length $t$ corresponds bijectively to a box with hook length $t$, the $t$-core $T(\lambda)$ of a partition $\lambda$ is itself a $t$-core.
\subsection{$t$-cores of partitions}
We will need restrictions of a bijection from \cite{GKS} to two subsets of $t$-cores. First, we recall this bijection. Let $\lambda$ be a $t$-core, we define the vector $ \phi(\lambda):=(n_0, n_1,\ldots, n_{t-1})$ as follows. We label the box $(i,j)$ of $\lambda$ by $(j-i)$ modulo $t$. We also label the boxes in the (infinite) column 0 in the same way, and we call the resulting diagram the \emph{extended t-residue diagram} (see Figure~\ref{fig3} below). A box is called \emph{exposed} if it is at the end of a row of the extended $t$-residue diagram. The set of boxes $(i,j)$ of the extended $t$-residue diagram satisfying $t(r-1)\leq j-i<tr$ is called a \emph{region} and labeled $r$. We define $n_i$ as the greatest integer $r$ such that the region labeled $r$ contains an exposed box with label $i$. 
\begin{figure}[h!]\begin{center}
\includegraphics[scale=1]{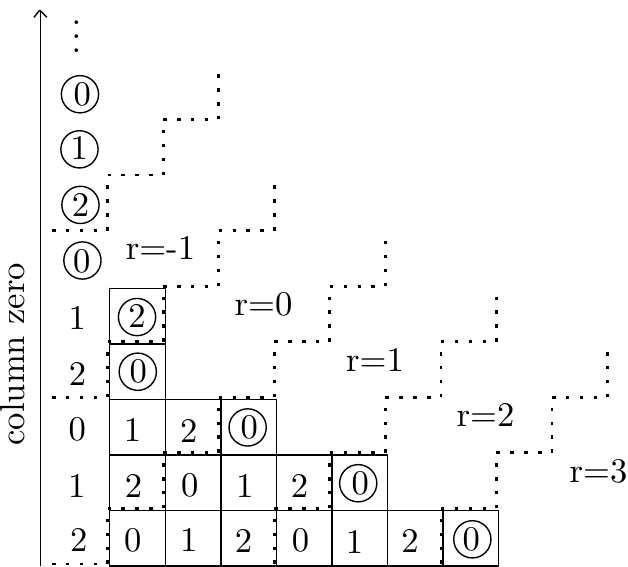}
\caption{\label{fig3}The extended 3-residue diagram of the $3$-core $\lambda=(7,5,3,1,1)$. The exposed boxes are circled.}
\end{center}
\end{figure}
\begin{theorem}[\cite{GKS}]\label{GKS}
The map $\phi$ is a bijection between $t$-cores and vectors of integers $ {\bf n}=(n_0,n_1,\ldots, n_{t-1})$ $ \in \mathbb{Z}^t$, satisfying $n_0+ \cdots +n_{t-1}=0$, such that: \begin{equation} |\lambda | = \frac{t\|{\bf n}\|^2}2 + {\bf b \cdot n}= \frac{t}{2}\sum_{i=0}^{t-1}n_i^2+ \sum_{i=0}^{t-1}in_i,\end{equation}
where~ ${\bf b } := (0,1,\ldots,t-1)$, $\|{\bf n}\|$ is the euclidean norm of ${\bf n}$, and ${\bf b \cdot n}$ is the scalar product of ${\bf b}$ and ${\bf n}$.
\end{theorem}

For example,  the $3$-core $\lambda=(7,5,3,1,1)$ of Figure~\ref{fig3} satisfies $\phi(\lambda)$ = $(3,-2,-1)$. We indeed have $7+5+3+1+1=17= |\lambda|= \frac{3}{2}(9+4+1)-2-2.$
\subsection{Self-conjugate $t$-cores}
Next we come to the definition of a subfamily of $\mathcal{P}_{(t)}$ which naturally appears in the proof of our type $\widetilde{C}$ formula. We define \emph{self-conjugate $t$-cores} as elements $\lambda$ in $\mathcal{P}_{(t)}$ satisfying $\lambda=\lambda^*$, where $\lambda^*$ is the conjugate of $\lambda$ (see \cite{EC}). We denote by $SC_{(t)}$ the set of self-conjugate $t$-cores and by $\lfloor t/2 \rfloor$ the greatest integer smaller or equal to $t/2$. 
\begin{proposition}\label{autoconjugue} There is a bijection $\phi_1$ between the partitions $\lambda \in SC_{(t)}$ and vectors of integers $ \phi_1(\lambda):={\bf n}\in \mathbb{Z}^{\lfloor t/2 \rfloor},$ such that: 
\begin{equation}|\lambda | = t\|{\bf n}\|^2 + {\bf c} \cdot {\bf n} , \quad \mbox{with~}
{\bf c } :=  \left\lbrace \begin{array}{lr}
    (1,3,\ldots,t-1) &\mbox{for $t$ even}, 
\vspace*{0.05cm}\\ (2,4,\ldots,t-1)&\mbox{for $t$ odd}.
  \end{array}\right. \end{equation}
\end{proposition}

Moreover, the image of a self-conjugate $t$-core $\lambda$ under $\phi_1$ is the vector whose components are the $\lfloor t/2 \rfloor$  last ones of $\phi(\lambda)$. This proposition is essentially proved in \cite{GKS}.

For example, the self-conjugate $3$-core $\lambda$ of Figure~\ref{fig1} satisfies $\phi(\lambda)=(3,0,-3)$; therefore its image under $\phi_1$ is the vector $(-3)$.

\subsection{$t$-cores of doubled distinct partitions} We will also need a second subfamily of $\mathcal{P}_{(t)}$ in our proof of Theorem~\ref{theoremeintro}. Let $\mu^0 $ be a partition with distinct parts. We denote by S($\mu^0$) the shifted Ferrers diagram of $\mu^0$, which is its Ferrers diagram where for all $1\leq i \leq \ell(\mu^0)$, the $i^{th}$ row is shifted by $i$ to the right (see Figure~\ref{fig4} below). 
\begin{figure}[h!]\begin{center}
\includegraphics[scale=1]{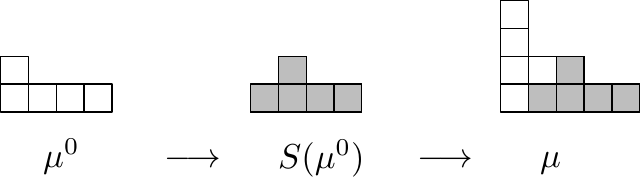}
\caption{\label{fig4}The construction of the doubled distinct partition $\mu $, for  $\mu^0=(4,1)$.}\end{center}
\end{figure}
\begin{definition}[\cite{GKS}] We define the \emph{doubled distinct partition} $\mu$ of $\mu^0$ as the partition whose Ferrers diagram is obtained by adding $\mu^0_i$ boxes to the  $i^{th}$ column of S($\mu^0$) for all $1\leq i \leq \ell(\mu^0)$. We denote by $DD$ the set of doubled distinct partitions and by $DD_{(t)}$ the subset of $t$-cores in $DD$. 
\end{definition}

\begin{proposition}\label{DD}There is a bijection $\phi_2$ between the partitions $\mu \in DD_{(t)}$ and vectors of integers $ \phi_2(\mu) :={\bf n}\in \mathbb{Z}^{\lfloor (t-1)/2 \rfloor},$ such that:  \begin{equation} |\mu | = t\|{\bf n}\|^2 + {\bf d}\cdot {\bf n} ,
\quad \mbox{with~} {\bf d}:=\left\lbrace\begin{array}{ll}
(2,4,\ldots,t-2)&\mbox{for $t$ even},\vspace*{0.05cm}
\\ (1,3,\ldots,t-2)&\mbox{for $t$ odd}.\end{array}\right.\end{equation}
\end{proposition}
Besides, the image of a doubled distinct $t$-core $\mu$ under $\phi_2$ is the vector whose components are the $\lfloor (t-1)/2 \rfloor$ last ones of $\phi(\mu)$. Again, Proposition~\ref{DD} is essentially proved in~\cite{GKS}.
~\\For example, the doubled distinct $3$-core  $\mu=(5,3,1,1)$ of Figure~\ref{fig4}, right,  satisfies $\phi(\mu)=(0,2,-2)$; so its image under $\phi_2$ is the vector $(-2)$.

\subsection{Generating function of $SC_{(t)} \times DD_{(t)}$}
We will now focus on pairs of $t$-cores in the set $SC_{(t)} \times DD_{(t)}$. We can in particular compute the generating function of these objects. Let  $(\lambda,\mu)$ be an element of $SC_{(t)} \times DD_{(t)}$. We define the weight of $(\lambda,\mu)$ as
$|\lambda|+|\mu|$, and we denote by $ h_t$ the generating function
\begin{equation}h_t(q) :=\sum_{(\lambda,\mu) \in SC_{(t)} \times DD_{(t)}} q^{|\lambda|+|\mu|}.\end{equation}
We would like to mention that the first step towards discovering Theorem~\ref{theoremeintro} was the computation of the Taylor expansion of $h_3(q)$, whose first terms seemed to coincide with the ones in the generating function of the vectors of integers involved in $\eqref{equaC}$ for $t=2$.
\begin{proposition}
The following equality holds for any integer $t\geq 0$:
\begin{equation}
h_{t+1}(q) =\frac{(q^2;q^2)_\infty}{(q;q)_\infty} (q^{t+1};q^{t+1})_\infty (q^{2t+2};q^{2t+2})^{t-1}_\infty, \quad \mbox{where~} (q;q)_\infty := \prod_{j \geq 1} (1-q^j).
\end{equation}
\end{proposition}
\begin{proof} Both generating functions of $SC_{(t)}$ and $DD_{(t)}$ are already known (see~\cite{GKS}). The idea of the proof, that we will not detail, is to compute their product, which can be done by considering the parity of $t$. To conclude, it remains to use Sylvester's bijection between partitions with odd parts and partitions with distinct parts in order to simplify and unify the resulting expressions.\end{proof}

\section{A Nekrasov-Okounkov type formula in type $\widetilde{C}$}\label{section2} 
The goal of this section is to sketch the proof of Theorem~\ref{theoremeintro}.

The global strategy is the following: we start from Macdonald's formula \eqref{equaC} in type $\widetilde{C}_t$, in which we replace the sum over vectors of integers by a sum over pairs of $t+1$-cores, the first in $SC_{(t+1)}$, and the second in $DD_{(t+1)}$. To do this, we need a new bijection $\varphi$ satisfying some properties that we will explain. This will allow us to establish Theorem~\ref{thmprincipal} of Section~\ref{3.2} below for all integers $t \geq 2$. An argument of polynomiality will then enable us to extend this theorem to any complex number $t$. Then, a natural bijection between pairs $(\lambda,\mu)$ in $SC \times DD$, and doubled distinct partitions (with weight equal to $|\lambda|+|\mu|$) will allow us to conclude. Note that at this final step, the partitions need not be $t+1$-cores.

\subsection{The bijection $\varphi$}
In what follows, we assume that $t\geq 2$ is an integer.
\begin{definition}\label{deltai} If  $(\lambda, \mu)$ is a pair belonging to $SC_{(t+1)} \times DD_{(t+1)}$, we denote by $\Delta$ the \emph{set of principal hook lengths} of $\lambda$ and $\mu$, and for all $ i \in \{1,\ldots, t\}$ we define
\begin{equation}\Delta_i:= Max\left(\{ h \in \Delta , h \equiv \pm i -t-1 \mbox{~mod~}2t+2 \} \cup \{i-t-1\}\right).\end{equation}
\end{definition}
For example, for $\lambda= (7,5,3,2,2,1,1)$, $\mu = (5,3,1,1)$ and $t+1=3$, we have $\Delta=\{13,8,7,2,1\}$, $\Delta_1=8$, and $
\Delta_2 =13$ (see Figure~\ref{fig5}).
\begin{figure}[h!]\begin{center}
\includegraphics[scale=1.2]{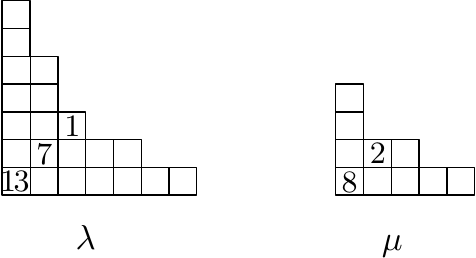}
\caption{\label{fig5}Computation of $\Delta$, $\Delta_1$ and $\Delta_2$ for a $(\lambda, \mu) \in SC_{(3)} \times DD_{(3)}$.}\end{center}
\end{figure}

As $\lambda$ (\emph{resp.} $\mu$) is self-conjugate (\emph{resp.} doubled distinct), all of its principal hook lengths are odd (\emph{resp.} even).  The knowledge of the set $\Delta$ enables us to reconstruct uniquely both partitions $\lambda$ and $\mu$. The following theorem shows that in fact, when these two partitions are $t+1$-cores, it is enough to know the $\Delta_i$'s to recover $\lambda$ and $\mu$ (so knowing the hook length maxima in each congruency class modulo $2t+2$ is enough).
\begin{theorem}\label{phi} Set ${\bf e}:=(1,2,\ldots, t)$. 
There is a bijection $\varphi$ between $SC_{(t+1)} \times DD_{(t+1)}$ and $ \mathbb{Z}^t$ such that $\varphi(\lambda,\mu):={\bf n}=(n_1,\ldots, n_t)$ satisfies:
\begin{equation} |\lambda|+|\mu|=(t+1) \|{\bf n}\| ^2 +{\bf e} \cdot {\bf n}= (t+1) \sum_{i=1}^t \left(n_i^2+in_i\right).\end{equation}
Besides, the following relation holds for all integers $i \in \{1,\ldots, t\}$: 
\begin{equation}\label{liennideltai}t+1+\Delta_i= \sigma_i ((2t+2)n_i+i),\end{equation}
where $\sigma_i$ is equal to $1 $ (\emph{resp.} $-1$) if $n_i \geq 0$ (\emph{resp.} $n_i<0$).
\end{theorem}
We can be more explicit about the construction of $\varphi$. Recall the bijections $\phi_1$ and $\phi_2$ defined in Propositions~\ref{autoconjugue} and \ref{DD}. If $(t+1)$ is odd, then $n_{2i}$ (\emph{resp. } $n_{2i+1}$) is the $i^{th}$ component of $\phi_1(\lambda)$ (\emph{resp.} $\phi_2(\mu)$); and if $(t+1)$ is even, $n_{2i}$ (\emph{resp. } $n_{2i+1}$) is the $i^{th}$component of $\phi_2(\mu)$ (\emph{resp.} $\phi_1(\lambda)$) .

It is then easy to prove that $\varphi$ is a bijection. The key property, which is hard to prove, is the one expressed in \eqref{liennideltai}; we do not give the proof here.

For example, the pair of $3$-cores  ($\lambda,\mu$) of Figure~\ref{fig5} satisfies $\varphi(\lambda,\mu)=(-2,-3)$. We have $31 =  |\lambda|+|\mu| =3(4+9)+1(-2)+2(-3)$.
Moreover, $\Delta_1=8$, $
\Delta_2 =13$. We verify that $3+\Delta_1=11=-\left(6n_1+1\right)$, and
$3+\Delta_2=16=-\left(6n_2+2 \right).$
\medskip
\\The inverse of $\varphi$ can be recursively described as follows. Fix a vector ${\bf n}=(n_1,\ldots, n_t)$ in $\mathbb{Z}^t$, then
 \begin{itemize}\itemsep-4pt
\item if all the $n_i$'s are equal to zero, then $\lambda$ and $\mu$ are empty,
\item if a $n_i$ is equal to $1$, then the corresponding partition ($\lambda$ or $\mu$, depending on the parity of $i$) contains a principal hook of length $t+1+i$,
\item if a $n_i$ is equal to $-1$, then the corresponding partition contains a principal hook of length $t+1-i$,
\item the preimage of $(n_1,\ldots,n_i +1, \ldots, n_t)$ if $n_i>0$ (\emph{resp.} $(n_1,\ldots,n_i -1, \ldots, n_t)$ if $n_i <0$) is the preimage of $(n_1,\ldots,n_i, \ldots, n_t)$ in which we add in the corresponding partition a principal hook of length $(t+1)(2n_i-1)+i$ (\emph{resp.} $(t+1)(-2n_i-1)-i$).
\end{itemize}
\begin{remark}\label{remarquedelta} There are three immediate consequences of the previous recursive description of $\varphi^{-1}$.
\\(i) There can not be in  $\Delta$ both a principal hook length equal to  $i+t+1$ mod $2t+2$ and a principal hook length equal to $-i+t+1$ mod $2t+2$.
\\(ii) If $h>2t+2$ belongs to  $\Delta$, then $h-2t-2$ also belongs to $\Delta$.
\\(iii) If a finite subset of $\mathbb{N}$ verifies the two former properties (i) and (ii) and does not contain any element equal to zero modulo $2t+2$, then it is the set $\Delta$ of a pair of $(t+1)$-cores $(\lambda,\mu)\in DD_{(t+1)} \times SC_{(t+1)}$.
\end{remark}
By using our bijection $\varphi$, and by setting $v_i = (2t+2)n_i +i$ for $1 \leq i \leq t$, we can replace the sum in Macdonald's formula~\eqref{equaC} by a sum over pairs $(\lambda,\mu) \in DD_{(t+1)} \times SC_{(t+1)}$ (and not over vectors of integers). Therefore \eqref{equaC} takes the form (recall that $\sigma_i$ is equal to $1 $ (\emph{resp.} $-1$) if $n_i \geq 0$ (\emph{resp.} $n_i<0$)):
\begin{eqnarray}\label{eqcoeff2} \displaystyle\prod\limits_{k \geq 1}(1-x^k)^{2t^2+t} & 
= \displaystyle c_1 \sum\limits_{\lambda, \mu} x^{|\lambda| +|\mu|}\prod\limits_i (2t+2n_i +i)\prod\limits_{i<j}((2t+2n_i +i)^ 2-(2t+2n_j +j)^2) 
\\ &=\displaystyle c_1 \sum\limits_{\lambda, \mu} x^{|\lambda| +|\mu|}\prod_i \sigma_i(t+1+\Delta_i)\prod_{i<j}((t+1+\Delta_i)^2-(t+1+\Delta_j)^2).\label{eqcoeff}\end{eqnarray}

\subsection{Simplification of coefficients}\label{3.2}
The next step towards proving Theorem~\ref{theoremeintro} is a simplication of both products on the right-hand side of \eqref{eqcoeff}, in such a way that they do not depend on the $\Delta_i$'s (and more generally, that they do not depend on congruency classes modulo $2t+2$). To do that, we need the following notion defined in \cite{HAN}, but only for odd integers.
 \begin{definition}A finite set of integers $A$ is a $2t+2$-\emph{compact set} if and only if the following conditions hold:
\begin{itemize}\itemsep-4pt
\item[(i)] $-1,-2, \ldots, -2t-1$ belong to $A$;
\item[(ii)] for all $a \in A$ such that $a \neq -1,-2, \ldots, -2t-1$, we have $a \geq 1$ and $a \not\equiv 0~\mbox{mod~} 2t+2$;
\item[(iii)] let $b>a \geq 1$ be two integers such that $a \equiv b$ mod $2t+2$. If $b \in A$, then $a \in A$.
\end{itemize}
\end{definition}
Let A be a $2t+2$-compact set. An element $a\in A$ is $2t+2$-maximal if for any element $b>a$ such that $a\equiv b $ mod $2t+2$, $b \notin A$ ($i.e.$ a is maximal in its congruency class modulo $2t+2$). The set of $2t+2$-maximal elements is denoted by $max_{2t+2}(A)$. It is clear by definition of compact sets that $A$ is uniquely determined by $max_{2t+2}(A)$. We can show the following lemma, whose proof is analogous to the one of \cite{HAN}, but in the even case.
\begin{lemma}\label{lemme han} For any $2t+2$-compact set A, we have:
\begin{equation}
-\prod_{a\in A, a>0}\left(1-\left(\frac{2t+2}{a} \right)^2\right) = \prod_{a \in max_{2t+2}(A)} \frac{a+2t+2}{a}.
\end{equation}
\end{lemma}
Now the strategy is to do an induction on the number of principal hooks of the pairs $(\lambda, \mu$) appearing in \eqref{eqcoeff}. The two following lemmas are the first step; their proofs are omitted due to their technical complexities and lengths.

Let $(\lambda, \mu)$ be in $SC_{(t+1)} \times DD_{(t+1)}$ with $\lambda$ or $\mu$ non empty,  and let $\Delta$ be the set of principal hook lengths of $\lambda$ and $\mu$, from which we can define the $\Delta_i$'s as in Definition~\ref{deltai}. We denote by $h_{11}$ the maximal element of $\Delta$. We denote by $(\lambda', \mu') \in SC_{(t+1)} \times DD_{(t+1)}$ the pair  obtained by deleting the principal hook of length $h_{11}$. We denote by $\Delta'$ the set of principal hook lengths of $\lambda'$ and $\mu'$, and consider its associated $\Delta_i'$'s. 
\begin{lemma}\label{lemme 1}If $i_0$ is the (unique) integer such that $\Delta_{i_0} = h_{11}$, then we have:
\begin{multline}\label{eqlemme 1}
\prod_i \frac{\sigma_i(t+1+\Delta_i)}{\sigma_i'(t+1+\Delta_i')}\prod_{i<j}\frac{(t+1+\Delta_i)^ 2-(t+1+\Delta_j)^2}{(t+1+\Delta_i')^ 2-(t+1+\Delta_j')^2}= \left(1-\frac{2t+2}{h_{11}}\right)\left(1-\frac{t+1}{h_{11}}\right)\\\times \left(\frac{h_{11}+t+1}{h_{11}-t-1}\right)\left(\frac{h_{11}}{h_{11}-2t-2}\right)\left(\frac{2h_{11}}{2h_{11}-2t-2}\right) \prod_{j \neq i_0} \frac{(h_{11}+ \Delta_j +2t+2)(h_{11}-\Delta_j)}{(h_{11}+ \Delta_j )(h_{11}-\Delta_j -2t-2)}.
\end{multline}
\end{lemma}
\begin{lemma}\label{lemme 2}
With the same notations as above, we define the set \begin{equation}E:=\bigcup \limits _{j \neq i_o} \{h_{11}+ \Delta_j, h_{11}-\Delta_j -2t-2\} \cup\{h_{11}-t-1, h_{11}-2t+2, 2h_{11}-2t+2\}. \end{equation} Then E is the $max_{2t+2}(H)$ of a unique $2t+2$-compact set $H$, which is independant of $t+1$. Moreover, its subset $H_{>0}$ of positive elements is made of elements of the form  $h_{11} + \tau_j j$, where $1\leq j\leq h_{11}-1$, and $\tau_j$ is equal to $1$ if $j$ is a principal hook length ($i.e.~j \in \Delta$) and to $-1$ otherwise. 
\end{lemma}

Now, we are able to derive the following lemma.
\begin{lemma}\label{lemme-recap}
If $(\lambda, \mu)$ is in $SC_{(t+1)} \times DD_{(t+1)}$ and $ (n_1,\ldots, n_t):= \varphi(\lambda, \mu)$, then the following equality holds:
\begin{eqnarray}\label{lemme4g}
\prod_i (2t+2n_i +i)\prod_{i<j}\left((2t+2n_i +i)^ 2-(2t+2n_j +j)^2\right)\hspace*{5cm}\\= \frac{\delta_\lambda \delta_\mu}{c_1}\prod_{h_{ii} \in \Delta} \left(1-\frac{2t+2}{h_{ii}}\right)\left(1-\frac{t+1}{h_{ii}}\right) \prod_{j=1}^{h_{ii}-1}\left( 1-\left(\frac{2t+2}{h_{ii} +\tau_j j} \right)^2\right)\label{lemme4d},
\end{eqnarray}
where $\delta_\lambda$ and $\delta_\mu$ are defined in Section~\ref{defs}.
\end{lemma}
\begin{proof}
Let us just show the idea of the proof. Starting from \eqref{lemme4g}, we transform by using $\varphi$ the products into products involving the $\Delta_i$'s, as we did for identifying \eqref{eqcoeff2} and \eqref{eqcoeff}. Next we do an induction on the cardinality of $\Delta$. To do this, we delete in $\lambda$ or $ \mu$ the hook corresponding to the largest element of $\Delta$, and we rewrite the product over the $\Delta_i$'s by using Lemma~\ref{lemme 1}. The successive right-hand sides of \eqref{eqlemme 1}, obtained by doing the induction, can be simplified into the products on the right-hand side of \eqref{lemme4d} by using Lemmas~\ref{lemme han} and \ref{lemme 2}. There are $D(\lambda)+D(\mu) $ steps in the induction, each of which giving rise to a minus sign. This explains the term $\delta_\lambda \delta_\mu$.  The base case corresponds to empty partitions $\lambda$ and $\mu$. In this case $\Delta_i =i-t-1$, $1\leq i\leq t$, therefore
\begin{eqnarray}
\prod_i \sigma_i(t+1+\Delta_i)\prod_{i<j}\left((t+1+\Delta_i)^ 2-(t+1+\Delta_j)^2\right)=\prod_i i \prod_{i<j}\left( i^2-j^2\right)= \frac{1}{c_1}. 
\end{eqnarray}
\end{proof}

We can finally prove the following result, which will be seen to be equivalent to Theorem~\ref{theoremeintro}.
\begin{theorem}\label{thmprincipal}
The following identity holds for any complex number t:
\begin{equation}\label{eqprincip}
\prod\limits_{n \geq 1}(1-x^n)^{2t^2+t} =\sum_{(\lambda,\mu)} \delta_\lambda \delta_\mu x^{|\lambda|+ |\mu|} \prod\limits_{h_{ii} \in \Delta} \left(1-\frac{2t+2}{h_{ii}}\right)\left(1-\frac{t+1}{h_{ii}}\right) \prod\limits_{j=1}^{h_{ii}-1} \left(1-\left(\frac{2t+2}{h_{ii} + \tau_j j} \right)^2 \right),
\end{equation}
where the sum ranges over pairs $(\lambda, \mu)$ of partitions,  $\lambda$ being self-conjugate and $\mu$ being doubled distinct.
\end{theorem}
\begin{proof}
Thanks to Macdonald's formula~\eqref{equaC} and Lemma~\ref{lemme-recap}, equation~\eqref{eqprincip} holds if the sum on the right-hand side is over pairs $(\lambda, \mu) \in SC_{(t+1)} \times DD_{(t+1)}$ and if $t$ is a positive integer. We will show that the product 
\begin{equation} Q :=\prod_{h_{ii} \in \Delta}\left( 1-\frac{2t+2}{h_{ii}}\right)\left(1-\frac{t+1}{h_{ii}}\right) \prod_{j=1}^{h_{ii}-1} \left(1-\left(\frac{2t+2}{h_{ii} +\tau_j j} \right)^2 \right)
\end{equation} vanishes if the pair $(\lambda, \mu)$ is not a pair of $t+1$-cores. Indeed, set $(\lambda, \mu) \in SC \times DD$, and let $\Delta$ be the set of principal hook lengths of $\lambda$ and $\mu$. We show that $Q$ vanishes unless $\Delta$ verifies the three hypotheses of $(iii)$ in Remark~\ref{remarquedelta}. Assume $Q \neq 0$.

First, let $h_{ii}>2t+2$  be an element of $\Delta$.  If $h_{ii}-2t-2$  was not a principal hook length, then the term corresponding to $j=h_{ii}-2t-2$ in the second product of $Q$ would vanish by definition of $\tau_j$. So $(ii)$ is satisfied.

Second, let $k, k'$ be nonnegative integers. If $(2k+1)(t+1)+i$ and $(2k'+1)(t+1)-i$ both belonged to $\Delta$, then by induction and according to the previous case, the product $Q$ would vanish if $t+1+i$ and $t+1-i$ did not belong to $\Delta$. But if $t+1+i$ and $t+1-i$ belonged to $\Delta$, the term $1-\left(\frac{2t+2}{(t+1+i)+(t+1-i)}\right)^2$ would vanish. So $(2k+1)(t+1)+i$ and $(2k'+1)(t+1)-i$ can not be both principal hook lengths if $Q$ is nonzero. So $(i)$ is satisfied.

Third, if $\Delta$ contains multiples of $t+1$, we denote by $h_{ii}$ the smallest such principal hook length. If  $h_{ii}= t+1$ or $h_{ii}=2t+2$, then the first term of the product $Q$ would vanish. Otherwise, $h_{ii}-2t-2$ does not belong to $\Delta$ by minimality, and so the term corresponding to $j=h_{ii}-2t+2$ in the second product of $Q$  would vanish.

 According to Remark~\ref{remarquedelta}, if $Q \neq 0$, then  $(\lambda, \mu)$ is a pair of $(t+1)$-cores. So formula \eqref{eqprincip} remains true for any positive integer $t$ if the sum ranges over $SC \times DD$. To conclude, we give a polynomiality argument which generalizes \eqref{eqprincip} to all complex numbers $t$. To this aim, we can use the following formula:
\begin{equation}
\prod_{k \geq 1} \frac{1}{1-x^k} = \mbox{exp}\left(\sum_{k \geq 1} \frac{x^k}{k(1-x^k)}\right),
\end{equation}
in order to rewrite the left-hand side of~\eqref{eqprincip} in the following form:
\begin{equation}\label{cotegauche}
\mbox{exp}\left(-(2t ^2+t)\sum_{k \geq 1} \frac{x^k}{k(1-x^k)}\right).
\end{equation}
Let $m$ be a nonnegative integer. The coefficient $C_m(t)$ of $x^m$ on the left-hand side of~\eqref{eqprincip} is a polynomial in $t$, according to~\eqref{cotegauche}, as is the coefficient $D_m(t)$ of $x^m$ on the right-hand side. Formula~\eqref{eqprincip} is true for all integers $t \geq 2$, it is therefore still true for any complex number $t$.
\end{proof}

Let $(\lambda, \mu)$ be in $SC \times DD$, with set of principal hook lengths $\Delta$. We denote by $2\Delta$ the set of elements of $\Delta$ multiplied by $2$. Note that we can uniquely associate to $(\lambda, \mu)$ a partition $\nu \in DD$ with set of principal hook lengths $2\Delta$.
\begin{theorem}\label{bijcouple}
The partition $\nu$ satisfies $|\lambda|+|\mu| = |\nu|/2$, $\delta_\lambda \delta_\mu = \delta_\nu$, and:
\begin{align}\prod\limits_{h_{ii} \in \Delta} \left(1-\frac{2t+2}{h_{ii}}\right)\left(1-\frac{t+1}{h_{ii}}\right) \prod\limits_{j=1}^{h_{ii}-1}\left(  1-\left(\frac{2t+2}{h_{ii} +\tau_j j} \right)^2 \right)= \prod\limits_{h \in \nu} \left(1-\frac{2t+2}{h \, \varepsilon_h}\right)\label{eqbij},
\end{align}
 where $\varepsilon_h$ is equal to $-1$ if $h$ is the hook length of a box  strictly above the principal diagonal, and to $1$ otherwise.
\end{theorem}
We omit the proof here; the difficult point  being \eqref{eqbij}, whose proof uses an induction on the number of principal hooks.
With Theorems~\ref{thmprincipal} and \ref{bijcouple}, Theorem~\ref{theoremeintro} straightforwardly follows.

\section{Some applications}\label{section3}
We give here some of the many applications of Theorem~\ref{theoremeintro}.  First, taking $t=-1$ in \eqref{eqtheoremeintro} yields the following famous expansion, where the sum ranges over partitions with distinct parts:
\begin{equation} 
\prod_{n \geq 1}(1-x^n) = \sum_\lambda (-1)^{\#\{\mbox{parts~of~}\lambda\}} x ^{|\lambda|}.
\end{equation}

Second, from \eqref{eqtheoremeintro}, \eqref{nekrasov} and the classical hook formula (see for instance \cite{EC}), we derive its following symplectic analogue, valid for any positive integer $n$:
\begin{equation}
\sum_{\stackrel{\lambda \in DD}{|\lambda|= 2n}} \prod_{h \in \mathcal{H}(\lambda)}\frac{1}{h} = \frac{1}{2^n n!}.
\end{equation}

Finally, we can prove the following theorem, which is surprising regarding the right-hand sides of the formulas, and which establishes a link between Macdonald's formulas in types $\widetilde{C}$, $\widetilde{B}$, and $\widetilde{BC}$.

\begin{theorem}\label{equi}
The following families of formulas are all generalized by Theorem~\ref{theoremeintro}:
\begin{itemize}\itemsep-2.5pt
\item[(i)]Macdonald's formula \eqref{equaC} in type $\widetilde{C}_t$ for any integer $t \geq 2$;
\item[(ii)]Macdonald's formula in type $\widetilde{B}_t$ for any integer $t \geq 3$:
\begin{equation}\label{equaB}
\eta(x)^{2t^2+t} = c_1 \sum_{{\bf v}} x^{\|{\bf v}\|^ 2/8(2t-1)} \prod_i v_i \prod_{i<j}(v_i^2-v_j^2) , 
\end{equation}
where the sum ranges over $t$-tuples ${\bf v} :=(v_1,\ldots, v_t) \in \mathbb{Z}^t$ such that $v_i \equiv 2i-1 \mbox{~mod~} 4t-2$ and $v_1+\cdots+v_t= t^2 \mbox{~mod~} 8t-4$;
\item[(iii)]Macdonald's formula in type $\widetilde{BC}_t$ for any integer $t\geq 1$:\end{itemize}
\begin{equation}\label{equaBC}
\eta(x)^{2t^2-t} =c_2 \sum_{{\bf v}} x^{\|{\bf v}\|^2 /8(2t+1)} (-1)^{(v_1+\cdots+v_t-t)/2}\prod_{i<j}(v_i^2-v_j^2) , ~\mbox{with~} c_2:= \frac{(-1)^{(t-1)/2}t!}{1! 2! \cdots (t-1)!},
\end{equation}
\hspace*{0.9cm}where the sum ranges over $t$-tuples ${\bf v} :=(v_1,\ldots, v_t) \in \mathbb{Z}^t$ such that $v_i \equiv 2i-1 \mbox{~mod~} 4t+2$.

\end{theorem}
\begin{proof} Here we do not give details of the proof (due to its length) and just present the ideas. By substituting $u := -t-1/2$  in \eqref{eqtheoremeintro}, and considering the positive integral values of $u$, we first prove that the product on the right-hand side vanishes for all partitions $\lambda$, except for those such that $\lambda$ does not contain a hook length equal to $2u-1$ for boxes strictly above the diagonal. By using some lemmas analogous to Lemmas~\ref{lemme han}--\ref{lemme-recap} (in the reverse sense) and a bijection analogous to $\varphi$, we manage to derive Macdonald's formula in type $\widetilde{B}_u$ for any integer $u \geq 3$. 
The same reasoning applies for type $\widetilde{BC}_\ell$ by doing the substitution $ \ell := t-1/2$ for integers $\ell \geq 1$. The partitions $\lambda$ that occur here are $2\ell+1$-cores.
\end{proof}

A natural question, and to which we were not able to answer, which arises is the following: is there a generalization analogous to Theorem~\ref{theoremeintro} for type $\widetilde{D}$?

\bibliographystyle{amsplain}

\end{document}